\documentclass[12pt]{amsart}
\usepackage{graphicx}
\usepackage{amssymb}
\usepackage{color}
\newtheorem{thm}{Theorem}[section]
\newtheorem{cor}[thm]{Corollary}
\newtheorem{lem}[thm]{Lemma}
\newtheorem{prop}[thm]{Proposition}
\theoremstyle{definition}
\newtheorem{defn}[thm]{Definition}
\theoremstyle{remark}

\numberwithin{equation}{section}
\newtheorem{exa}[thm]{Example}

\newcommand{\h}{\mathcal{H}}

\newcommand{\K}{\mathcal{K}}

\begin{document}
\title[Disjointness of continuous g-frames]{Disjointness of continuous g-frames and Riesz-type continuous g-frames }%
\author {Y. khedmati and M. R. Abdollahpour$^*$ }%
\address{
\newline
\indent Yavar Khedmati and Mohammad Reza Abdollahpour
\newline
\indent Department of Mathematics
\newline
\indent Faculty of Sciences
\newline \indent   University of Mohaghegh Ardabili
\newline \indent  Ardabil 56199-11367
\newline \indent Iran}
\email{khedmati.y@uma.ac.ir, khedmatiy.y@gmail.com}
\email{m.abdollah@uma.ac.ir, mrabdollahpour@yahoo.com}

\thanks{{\scriptsize
\hskip -0.4 true cm MSC(2010): Primary 41A58, 42C15.
$^*$Corresponding author
}}

\begin{abstract}
In this paper we introduce concepts of disjoint, strongly disjoint and weakly disjoint continuous $g$-frames in Hilbert spaces and we get some equivalent conditions to these notions. We also construct a continuous g-frame by disjoint continuous g-frames. Furthermore, we provide some results related to the Riesz-type continuous $g$-frames.
\end{abstract}
\maketitle
\section{{\textbf Introduction}}
In 1952, the concept of frames for Hilbert spaces was defined by Duffin and Schaeffer \cite{DS}. Frames are important tools in the signal processing, image processing, data compression, etc. Let $\h$ be a separable Hilbert space. We call a sequence $F=\{f_i\}_{i\in I} \subseteq \h$ a frame for $\h$ if there exist two constant $A_F, B_F> 0$ such that
\begin{eqnarray}\label{abc}
A_F \|f\|^2\leq\sum_{i\in I}|\langle f,f_i\rangle |^2 \leq B_F \|f\|^2,\quad f\in \h.
\end{eqnarray}
If in (\ref{abc}), $A_F=B_F=1$ we say that $F=\{f_i\}_{i\in I}$ is a Parseval frame for $\h$. Let $F=\{f_i\}_{i\in I}$ be a frame for $\h,$ then the operator
\begin{eqnarray*}
T_F:l_2(I) \rightarrow \h,\quad T_F(\{c_i\}_{i\in I})=\sum_{i\in I} c_if_i,
\end{eqnarray*}
is well define and onto, also its adjoint is
\begin{eqnarray*}
T^*_F:\h\rightarrow l_2(I) ,\quad T_F^* f=\{\langle f,f_i \rangle\}_{i\in I}.
\end{eqnarray*}
The operators $T_F$ and $T^*_F$ are called the synthesis and analysis operators of frame $F.$
The concepts of disjoint frames and strongly disjoint frames introduced by Han and Larson \cite{HanL}.
\begin{defn}
Let $F=\{f_i\}_{i\in I}$ and $G=\{g_i\}_{i\in I}$ be frames for Hilbert spaces $\h$ and $\K,$ respectively. We say that $F$ and $G$ are
\begin{enumerate}
\item[(i)]
Disjoint, if $\{f_i \oplus g_i\}_{i\in I}$ is a frame for $\h\oplus \K$.
\item[(ii)]
Strongly disjoint, if there are invertible operator $L_1 \in B(\h)$ and $L_2 \in B(\K)$ such that $\{L_1 f_i\}_{i\in I},\{L_2 g_i\}_{i\in I}$ and $\{L_1 f_i\oplus L_2g_i\}_{i\in I}$ are respective Parseval frames for $\h,\K$ and $\h \oplus \K$.
\end{enumerate}
\end{defn}
\begin{prop}\cite{HanL}
Let $F=\{f_i\}_{i\in I}$ and $G=\{g_i\}_{i\in I}$ be frames for Hilbert spaces $\h$ and $\K,$ respectively. Then
\begin{enumerate}
\item[(i)]
$F$ and $G$ are disjoint if and only if $RangeT^*_F \cap RangeT^*_G=\{0\}$ and $RangeT^*_F+RangeT^*_G$ is closed subspace of $l_2(I).$
\item[(ii)]
F and G are strongly disjoint if and only if $RangeT^*_F$ and $RangeT^*_G$ are orthogonal.
\end{enumerate}
\end{prop}
 In 1993, Ali, Antoine and Gazeau developed the notion of ordinary frame to a family indexed by a measurable space which are known as continuous frames \cite{Ali2}.
 \begin{defn}
 Let $\h$ be a complex Hilbert space and $(\Omega ,\mu)$ be a
measure space. The mapping $F:\Omega\to \h$ is called a continuous frame
 if
\begin{enumerate}
\item[(i)] $F$ is weakly-measurable, i.e., for all $f\in \h$,
$\omega\to\langle f,F({\omega})\rangle$
is a measurable function on $\Omega$,
\item[(ii)]
there exist constants $A_F, B_F> 0$ such that
\begin{eqnarray*}\label{deframe}
A_F \|f\|^{2}\leq \int_{\Omega}|\langle
f,F({\omega})\rangle|^{2}\,d\mu(\omega)\leq B_F \|f\|^{2}, \quad
f\in \h.
\end{eqnarray*}
\end{enumerate}
\end{defn}
If $F:\Omega\to \h$ is a continuous frame then the operator $S_F:\h\rightarrow \h$ defined
by $$\langle S_F f,g\rangle=\int_{\Omega}\langle
f,F({\omega})\rangle\langle F({\omega}),g\rangle d\mu(\omega), \quad f,g\in \h,$$ is positive and invertible. $S_F$ is called the continuous frame operator of $F.$ \par
 In 2006, $g$-frames or generalized frames introduced by Sun \cite{ws}. Abdollahpour and Faroughi introduced and investigated continuous $g$-frames and Riesz-type continuous $g$-frames  \cite{ostad}.
 Disjointness notions were developed to continuous frames by Gabardo and Han \cite{Gb} and to $g$-frames by Abdollahpour \cite{ostad2}.
In the rest of this paper we assume that $\h$ and $\K$ are complex Hilbert spaces and $(\Omega,\mu)$ is a measure space with positive measure $\mu$ and $\{\K_\omega:\omega\in \Omega\}$ is a family of Hilbert spaces. Now, we summarize some
 facts about continuous g-frames from \cite{ostad}. \par
We say that $F\in\prod_{\omega\in \Omega}\K_{\omega}$ is
strongly measurable if $F$ as a mapping of $\Omega$ to
$\bigoplus_{\omega\in \Omega}\K_{\omega}$ is measurable, where
$$\prod_{\omega\in \Omega}\K_{\omega}=\left\{f:\Omega\rightarrow\bigcup_{\omega\in \Omega}\K_{\omega}
:f(\omega)\in \K_{\omega}\right\}.$$
\begin{defn}
We say that $\Lambda=\{\Lambda_{\omega}\in
B(\h,\K_{\omega}):\,\omega\in\Omega\}$ is a continuous $g$-frame for $\h$  with respect to
$\{\K_{\omega}:\omega\in \Omega\}$ if
\begin{enumerate}
\item[(i)] for each $f\in \h$, $\{\Lambda_{\omega}f:\omega\in \Omega\}$ is
strongly measurable,
\item[(ii)]
there are two constants $0<A_\Lambda\leq
B_\Lambda<\infty$ such that
\begin{equation}\label{cgframe}
A_\Lambda\|f\|^{2}\leq\int_{\Omega}\|\Lambda_\omega f\|^{2}
d\mu(\omega)\leq B_\Lambda\|f\|^{2},\; f\in \h.
\end{equation}
\end{enumerate}
We call $A_\Lambda,B_\Lambda$ the lower and upper continuous $g$-frame bounds,
respectively.
$\Lambda$ is called a tight continuous $g$-frame if $A_\Lambda=B_\Lambda,$ and a
Parseval continuous $g$-frame if $A_\Lambda=B_\Lambda=1.$ If for each $\omega\in\Omega,$
$\K=\K_{\omega},$ then $\Lambda$ is called a continuous $g$-frame
with respect to $\K$. $\Lambda=\{\Lambda_{\omega}\in
B(\h,\K_{\omega}):\,\omega\in\Omega\}$ is called a continuous
$g$-Bessel family  if the right hand inequality in (\ref{cgframe})
holds for all $f\in \h$. In this case, $B_\Lambda$ is called the Bessel
constant.
\end{defn}
If there is no confusion, we use continuous $g$-frame (continuous $g$-Bessel family) instead of continuous $g$-frame for $\h$ with respect to $\{\K_\omega:\omega\in\Omega\}$  (continuous $g$-Bessel family for $\h$ with respect to
$\{\K_{\omega}:\omega\in \Omega\}$).
\begin{prop}\label{s}\cite{ostad}
Let $\Lambda=\{\Lambda_{\omega}\in B(\h,\K_{\omega}):\omega \in \Omega\}$ be a continuous
$g$-frame. Then there exists a unique positive and
invertible operator $S_\Lambda:\h\rightarrow \h$ such that
\begin{eqnarray*}\label{di}
\langle S_\Lambda f,g\rangle =\int_{\Omega}\langle f,\Lambda^{*}_{\omega}
\Lambda_{\omega}g\rangle d\mu(\omega),\quad f,g\in \h,
\end{eqnarray*}
 and $A_\Lambda I\leq S_\Lambda\leq B_\Lambda I.$
\end{prop}
The operator $S_\Lambda$ in Proposition \ref{s} is called the continuous
$g$-frame operator of $\Lambda.$ Also, we have
\begin{equation}\label{abcd}
\langle f,g \rangle =\int_\Omega \langle S_\Lambda^{-1}f,\Lambda_\omega^*\Lambda_\omega g
\rangle \,d\mu(\omega)
=\int_\Omega \langle f,\Lambda_\omega^*\Lambda_\omega S_\Lambda^{-1} g
\rangle\, d\mu(\omega),
\end{equation}
 for all $f,g\in \h.$ \\ We consider the space
$$\widehat {\K}=\left\{F\in \prod_{\omega\in\Omega}\K_{\omega}:
\text{F is strongly measurable},\:
\int_{\Omega}\|F(\omega)\|^{2}d\mu(\omega)<\infty\right\}.$$
It is clear that $\widehat {\K}$ is a Hilbert space with point wise operations and with the inner
product given by
$$\langle F,G \rangle=\int_{\Omega}\langle
F(\omega),G(\omega)\rangle d\mu(\omega),\quad F,G\in \widehat {\K}.$$
\begin{prop}\cite{ostad}\label{combi}
Let $\Lambda=\{\Lambda_{\omega}\in B(\h,\K_{\omega}):\omega \in \Omega\}$ be a continuous
$g$-Bessel family. Then the mapping $T_\Lambda:\widehat {\K}\rightarrow \h$
defined by
\begin{eqnarray}\label{ti}
\langle T_\Lambda F,g\rangle=\int_{\Omega}\langle\Lambda_{\omega}^{*}F(\omega),g\rangle
d\mu(\omega),\;F\in\widehat {\K},\,g\in\h,
\end{eqnarray}
is linear and bounded with $\|T_{\Lambda}\|\leq\sqrt{B_\Lambda}.$ Also, for
each $g\in \h$ and $\omega\in\Omega,$
$$(T_\Lambda^{*}g)(\omega)=\Lambda_{\omega}g.$$
\end{prop}
\begin{thm}\cite{ostad}\label{ctf}
Let $(\Omega,\mu)$ be a measure space, where $\mu$ is
$\sigma$-finite. Suppose that $\Lambda=\{\Lambda_{\omega}\in
B(\h,\K_{\omega}):\omega \in \Omega\}$ is a family of operators such
 $\{\Lambda_{\omega}f:\omega\in\Omega\}$ is strongly measurable, for each $f\in \h.$
 Then $\Lambda$ is a continuous $g$-frame if and only if the operator
$T_\Lambda:\widehat {\K}\rightarrow \h$
defined by (\ref{ti}) is bounded and onto.
\end{thm}
The operators $T_\Lambda$ and $T_\Lambda^{*}$ in Theorem \ref{ctf} are called the synthesis
and analysis operators of $\Lambda$, respectively.
\begin{defn}
Let $\Lambda=\{\Lambda_{\omega}\in B(\h,\K_{\omega}):\omega \in \Omega\}$ and $\Theta=\{\Theta_{\omega}\in B(\h,\K_{\omega}):\omega \in \Omega\}$  be  two continuous $g$-frames such that
$$\langle f,g\rangle=\int_{\Omega}\langle f,
\Theta^{*}_{\omega}\Lambda_{\omega}g\rangle d\mu(\omega),\; f,g\in\h,$$ then $\Theta$ is called a dual continuous $g$-frame of
$\Lambda.$
\end{defn}
Let $\Lambda=\{\Lambda_{\omega}\in B(\h,\K_{\omega}):\omega \in \Omega\}$ be a continuous $g$-frame. Then
$\widetilde\Lambda=\{\Lambda_\omega S_\Lambda^{-1}\in B(\h, \K_\omega):\omega \in \Omega\}$ is a continuous
$g$-frame and by (\ref{abcd}), $\widetilde\Lambda$ is a dual of $\Lambda$ and we call
$\widetilde\Lambda$ the canonical dual of $\Lambda$. One can always
get a tight continuous $g$-frame from any continuous $g$-frame, in fact, if $\Lambda=\{\Lambda_{\omega}\in B(\h,\K_{\omega}):\omega \in \Omega\}$ is a continuous $g$-frame then $\{\Lambda_\omega S_\Lambda^{-1/2}\in B(\h, \K_\omega):\omega \in \Omega\}$ is a Parseval continuous $g$-frame.
\par Two continuous $g$-Bessel families $\Lambda=\{\Lambda_{\omega}\in B(\h, \K_{\omega}):\omega \in \Omega\}$ and $\Theta=\{\Theta_{\omega}\in B(\h, \K_{\omega}):\omega \in \Omega\}$ are weakly equal, if for all $f\in H$,
\begin{eqnarray*}
\Lambda_\omega f=\Theta_\omega f, \quad a.e.\quad\omega\in\Omega.
\end{eqnarray*}
\par If the continuous $g$-frame
$\Lambda=\{\Lambda_{\omega}\in B(\h, \K_{\omega}):\omega \in \Omega\}$ have only one
dual (weakly), i.e., every dual of $\Lambda$ is weakly equal to the
canonical dual of $\Lambda$, then $\Lambda$ is called a Riesz-type continuous
$g$-frame.
\begin{thm}\cite{ostad}\label{riezs}
Let $\Lambda=\{\Lambda_{\omega}\in B(\h, \K_{\omega}):\omega \in \Omega\}$ be a continuous
$g$-frame. Then $\Lambda$ is a Riesz-type continuous $g$-frame if and only
if $Range
T^{*}_{\Lambda}=\widehat {\K}.$
\end{thm}
We mention that the authors of this paper studied some properties of continuous $g$-frames and Riesz-type continuous $g$-frames in \cite{abkh}.
\section{{\textbf Disjointness of continuous $g$-frames }}
In this section we study disjointness, strongly disjointness, weakly disjointness for continuous $g$-frames. We prove some results concern with these concepts and we construct a continuous $g$-frame by disjoint and strongly disjoint continuous $g$-frames.
\begin{defn}
Let $\Lambda=\{\Lambda_\omega \in B(\h, \K_{\omega}) : \omega \in \Omega \}$ and $\Theta=\{\Theta_\omega \in B(\K,\K_\omega) : \omega \in \Omega \}$ be two  continuous $g$-frames. Then $\Lambda$ and $\Theta$ are called:
\begin{enumerate}
\item[(i)] Strongly disjoint, if $RangeT_{\Lambda}^* \bot RangeT_{\Theta}^*.$
\item[(ii)] Disjoint, if $RangeT_{\Lambda}^*\cap RangeT_{\Theta}^*=\{0\}$ and $RangeT_{\Lambda}^*+RangeT_\Theta^*$ is a closed subspace of $\widehat {\K}.$
\item[(iii)] Complementary pair, if $RangeT^*_\Lambda \cap RangeT_\Theta^*=\{0\}$ and $RangeT^*_\Lambda+RangeT_\Theta^*=\widehat {\K}.$
\item[(iv)] Strongly complementary pair, if $RangeT_\Lambda^*\oplus RangeT_\Theta^*=\widehat {\K}.$
\item[(v)] Weakly disjoint, if $RangeT^*_\Lambda \cap RangeT_\Theta^*=\{0\}.$
\end{enumerate}
\end{defn}
\begin{thm}\label{diss}
Let $(\Omega, \mu)$ be a mesure space. Let $\Lambda=\{\Lambda_\omega \in B(\h,\K_\omega) : \omega \in \Omega \}$ and $\Theta=\{\Theta_\omega \in B(\K, \K_{\omega}) : \omega \in \Omega \}$ be two continuous $g$-frames. Consider $\Gamma=\{\Gamma_\omega \in B(\h\oplus \K,\K_\omega) : \omega \in \Omega \}$ where
\begin{align*}
\Gamma_\omega(h\oplus k)=\Lambda_\omega h+\Theta_\omega k,\quad\omega \in \Omega,h\in \h,k\in \K.
\end{align*}
Then $\Lambda$ and $\Theta$ are
\begin{enumerate}
\item[(i)] Strongly disjoint if and only if there exist invertible operators $L_1\in B(\h)$ and $L_2\in B(\K)$ such that $\{\Lambda_\omega L_ 1\in B(\h, \K_{\omega}) : \omega \in \Omega \}$ , $\{\Theta_\omega L_2 \in B(\K,\K_\omega) : \omega \in \Omega \}$ and $\{\Delta_\omega \in B(\h\oplus \K,\K_\omega) : \omega \in \Omega \}$ are Parseval continuous $g$-frames, where
\begin{align*}
\Delta_\omega(h\oplus k)=\Lambda_\omega L_1h+\Theta_\omega L_2 k,\quad \omega \in \Omega,h\in \h, k\in \K.
\end{align*}
\item[(ii)] Disjoint if and only if $\Gamma$ is a continuous $g$-frame.
\item[(iii)] Complementary pair if and only if $\Gamma$ is a Riesz-type continuous $g$-frame.
\item[(iv)] Strongly complementary pair if and only if they are strongly disjoint and $\Gamma$ is a Riesz-type continuous $g$-frame.
\item[(v)] Weakly disjoint if and only if
\begin{eqnarray*}
\{f\oplus g:\Gamma_\omega(f\oplus g)=0, \;\omega\in \Omega\}=\{0\}.
\end{eqnarray*}
\end{enumerate}
\end{thm}
\begin{proof}
 (i) For any $h\in \h$ and $k\in \K$ we have
 \begin{align*}
 \int_\Omega \|\Lambda_\omega S_{\Lambda}^{-1/2}h&+\Theta_\omega S_{\Theta}^{-1/2}k\|^2 d\mu(\omega)
 \\&=\int_\Omega \|\Lambda_{\omega }S_{\Lambda}^{-1/2}h \|^2 d\mu(\omega)+\int_\Omega \|\Theta_\omega S_{\Theta}^{-1/2}k \|^2 d\mu(\omega) 
 \\&+2Re\int_\Omega \big\langle\Lambda_\omega S_{\Lambda}^{-1/2}h,\Theta_\omega S_\Theta^{-1/2}k\big\rangle d\mu(\omega) 
 \\&=\int_\Omega \|\Lambda_\omega S_\Lambda^{-1/2}h \|^2 d\mu(\omega)+\int_\Omega \|\Theta_\omega S_\Theta^{-1/2}k \|^2 d\mu(\omega)
  \\&=\|h\|^2+\|k\|^2=\|h\oplus k\|^2.
 \end{align*}
 It is sufficient to take $L_1=S_\Lambda^{-1/2}$ and $L_2=S_\Theta^{-1/2}$.
 \\Conversly, for every $h\in \h$ and $k\in \K$ we have
 \begin{align*}
 \|h\oplus k\|^2&=\int_\Omega  \|\Delta_\omega(h\oplus k)\|^2 d\mu(\omega)
 \\&=\int_\Omega  \|\Lambda_\omega L_1h\|^2 d\mu(\omega)+\int_\Omega  \|\Theta_\omega L_2k\|^2 d\mu(\omega)
 \\&+2Re\int_\Omega \big\langle \Lambda_\omega L_1 h,\Theta_\omega L_2k\big\rangle d\mu(\omega)
 \\&=\|h\|^2+\|k\|^2+2Re\int_\Omega \big\langle \Lambda_\omega L_1 h,\Theta_\omega L_2k\big\rangle d\mu(\omega).
 \end{align*}
 Therefore
 \begin{align*}
 Re\int_\Omega \langle \Lambda_\omega L_1 h,\Theta_\omega L_2k\rangle d\mu(\omega)=0, \quad h\in\h,k\in\K.
 \end{align*}
 On the other hand
\begin{align*}
Im\int_\Omega \langle \Lambda_\omega L_1 h,\Theta_\omega L_2k\rangle d\mu(\omega)=-Re\int_\Omega \langle \Lambda_\omega L_1(ih),\Theta_\omega L_2k\rangle d\mu(\omega)=0.
\end{align*}
 Thus
 \begin{align*}
 \int_\Omega \langle \Lambda_\omega L_1 h,\Theta_\omega L_2k\rangle d\mu(\omega)=0,\quad h\in\h,k\in\K.
 \end{align*}
 Now, since $L_1$ and $L_2$ are invertible operators, $RangeT_\Lambda^*\bot RangeT_\Lambda^*.$
\\ (ii) Let $\Gamma$ be a continuous $g$-frame for $\h\oplus\K.$ Then for any $h\in \h, k\in \K$ we have
\begin{eqnarray}\label{yav1}
 A_\Gamma(\|h\|^2+\|k\|^2)\le \|T_{\Lambda}^* h+T_{\Theta}^* k\|^2  \le B_\Gamma(\|h\|^2+\|k\|^2).
\end{eqnarray}
Let there exist $h_1\in \h$ and $k_1\in \K$ such that $T_\Lambda^*h_1=T_\Theta^*k_1$. By the left hand inequality (\ref{yav1}), $h_1=k_1=0$ and so, $T_\Lambda^*h_1=T_\Theta^*k_1=0$. Consequently, $RangeT_{\Lambda}^*\cap RangeT_{\Theta}^*=\{0\}.$
 Also, $RangeT_{\Lambda}^*+RangeT_\Theta^*=Range T_\Gamma^*$ is a closed subspace of $\widehat {\K}$.
 \\Conversely, the operator
 \begin{align*}
 L:Range T_\Lambda^*\oplus Range T_\Theta^*&\rightarrow Range T_\Lambda^*+Range T_\Theta^*,\\ L(F\oplus G)&=F+G
 \end{align*}
 is a bijective bounded operator, In fact
 \begin{align*}
 \|L(F\oplus G)\|^2=\| F + G\|^2 &\le(\|F\|+\|G\|)^2
 \\&\le 2(\|F\|^2+\|G\|^2)=2\|F\oplus G\|^2.
 \end{align*}
Then for any $h\in \h,$ $k\in \K$ we have
 \begin{align*}
\|L^{-1}\|^{-2}.\min \{A_\Lambda,A_\Theta\}.\|h\oplus k\|^2 &\le \int_\Omega \|\Gamma_\omega(h\oplus k)\|^2d\mu(\omega)
\\&=\|L(T_\Lambda^*h\oplus T_\Theta^*k)\|^2
\\&\le \|L\|^2 .\max\{B_\Lambda,B_\Theta\}.\|h\oplus k\|^2.
 \end{align*}
 \\(iii)  Let $\Gamma$ be a Riesz-type continuous $g$-frame for $\h\oplus \K.$ Then by Theorem \ref{riezs}, we have
 \begin{eqnarray*}
 RangeT_\Lambda^*+RangeT_\Theta^*=RangeT_\Gamma^*=\widehat  {\K}.
 \end{eqnarray*}
 On the other hand, let $\phi\in RangeT_\Lambda^*\cap RangeT_\Theta^*$. Then there exist $h\in \h$ and $k\in \K$ such that $T_\Lambda^*h=T_\Theta^*k=\phi$. Therefore
 \begin{eqnarray*}
 T_\Gamma^*(h\oplus 0)=\phi=T_\Gamma^*(0\oplus k),
 \end{eqnarray*}
 since $T_\Gamma^*$ is one-to-one, $h=k=0$ and so $\phi=0.$
 \\Conversely, by the part (ii), $\Gamma$ is a continuous g-frame and
 \begin{eqnarray*}
 RangeT_\Gamma^*=RangeT_\Lambda^*+RangeT_\Theta^*=\widehat {\K}.
 \end{eqnarray*}
 So, by Theorem \ref{riezs}, $\Gamma$ is a Riesz-type continuous $g$-frame.
\\(iv) By applying the definition of strongly disjoint and (iii), the poof is completed.
\\(v) Let $\Lambda$ and $\Theta$ be weakly disjoint and $h\in\h,k\in\K$ such that $\Gamma_\omega(h\oplus k)=0$ for all $\omega\in\Omega,$
then
\begin{eqnarray*}
T_\Lambda^*h=T_\Theta^*(-k)\in RangeT_\Lambda^*\cap RangeT_\Theta^*=\{0\},
\end{eqnarray*}
and so, $h=k=0.$
Conversely, let $\phi\in RangeT_\Lambda^*\cap RangeT_\Theta^*.$ Then there exist $h\in\h,k\in\K$ such that $T_\Lambda^*h=T_\Theta^*k=\phi,$  therefore
\begin{eqnarray*}
T_\Gamma^*(h\oplus(-k))=T_\Lambda^* h-T_\Theta^* k=0.
\end{eqnarray*}
Hence $h=k=0,$ consequently, $\phi=0.$
\end{proof}
\begin{lem}\cite{cc4}\label{snew}
Suppose that $T : \K \rightarrow\h$ is a linear bounded, surjective operator. Then there exists a linear bounded operator (called the pseudo-inverse of $T$) $T^\dagger : \h\rightarrow \K$  for which $TT^\dagger f = f$ , for any $f\in\h.$
\end{lem}
By generalizing a result from \cite{newnew} we get a following proposition to construct a continuous $g$-frame from disjoint continuous $g$-frames.
\begin{prop}\label{ssnew}
Let $\Lambda=\{\Lambda_\omega\in B(\h,\K_\omega):\omega\in\Omega\}$ and $\Theta=\{\Theta_\omega\in B(\h,\K_\omega):\omega\in\Omega\}$ be two disjoint continuous $g$-frames and $L_1,L_2\in B(\h).$ If $L_1$ or $L_2$ is surjective, then $\Lambda L_1^*+\Theta L_2^*=\{\Lambda_\omega L_1^*+\Theta_\omega L_2^*\in B(\h,\K_\omega):\omega\in\Omega\}$ is a continuous $g$-frame.
\end{prop}
\begin{proof}
Let $L_1$ is surjective, then by Lemma \ref{snew}, there exist $L_1^\dagger\in B(\h)$ such that $L_1L_1^\dagger=I,$ so $(L_1^\dagger )^*L_1^*=I.$ Hence for any $h\in\h$ we have 
\begin{align*}
\|h\|=\|(L_1^\dagger )^*L_1^*h\|\leq\|L_1^\dagger\|\|L_1^*h\|,
\end{align*}
thus
\begin{align*}
\|L_1^*h\|\geq\frac{\|h\|}{\|L_1^\dagger\|}.
\end{align*}
Suppose $\Gamma_\omega(h\oplus k)=\Lambda_\omega h+\Theta_\omega k,$ for all $h,k\in\h$ and for all $\omega\in\Omega.$ By part (ii) of Theorem \ref{diss}, $\Gamma=\{\Gamma_\omega\in B(\h\oplus\h):\omega\in\Omega\}$ is a continuous $g$-frame. Therefore, for any $h\in\h$ we have
\begin{align*}
\frac{A_\Gamma}{\|L_1^\dagger\|^2}\|h\|^2
\leq A_\Gamma\|L_1^* h\|^2
&\leq A_\Gamma\big(\|L_1^* h\|^2+\|L_2^* h\|^2\big)
\\&= A_\Gamma\|L_1^* h\oplus L_2^* h\|^2 
\\&\leq\int_\Omega\|\Gamma_\omega(L_1^* h\oplus L_2^* h)\|^2 d\mu(\omega)
\\&\leq B_\Gamma\|L_1^* h\oplus L_2^* h\|^2
\\&\leq B_\Gamma\big(\|L_1^* h\|^2+\|L_2^* h\|^2\big)
\\&\leq 2B_\Gamma.\max\{\|L_1\|^2,\|L_2\|^2\}\|h\|^2.
\end{align*}
Therefore we have 
\begin{align*}
\frac{A_\Gamma}{\|L_1^\dagger\|^2}\|h\|^2&\leq \int_\Omega\|(\Lambda_\omega L_1+\Theta_\omega L_2 )h\|^2 d\mu(\omega)
\\&\leq 2B_\Gamma.\max\{\|L_1\|^2,\|L_2\|^2\}\|h\|^2,\quad h\in\h.
\end{align*}
The proof is similar whenever $L_2$ is surjective.
\end{proof}
\begin{cor}
Let $\Lambda=\{\Lambda_\omega\in B(\h,\K_\omega):\omega\in\Omega\}$ and $\Theta=\{\Theta_\omega\in B(\h,\K_\omega):\omega\in\Omega\}$ be two disjoint continuous $g$-frames. Then $\Lambda+\Theta=\{\Lambda_\omega+\Theta_\omega\in B(\h,\K_\omega):\omega\in\Omega\}$ is a continuous $g$-frame.
\end{cor}
\begin{proof}
By considering $L_1=L_2=I$ in Proposition \ref{ssnew} the proof is completed.
\end{proof}
In the following results, we construct a continuous $g$-frame by strongly disjoint continuous $g$-frames by generalizing a result from \cite{HanL}.
\begin{prop}
Let $\Lambda=\{\Lambda_\omega \in B(\h, \K_{\omega}) : \omega \in \Omega \}$ and $\Theta=\{\Theta_\omega \in B(\h, \K_{\omega}) : \omega \in \Omega \}$ be strongly disjoint continuous $g$-frames and let $L_1, L_2\in B(\h)$ such that $L_1^* L_1+L_2^* L_2=AI $ for some $A>0$. Then $\Lambda L_1+\Theta L_2=\{\Lambda_\omega L_1+\Theta_\omega L_2\in B(\h,\K_\omega) : \omega\in\Omega\}$ is a continuous $g$-frame. In particular, $\alpha\Lambda+\beta\Theta=\{\alpha\Lambda_\omega+\beta\Theta_\omega\in B(\h,\K_\omega) : \omega \in \Omega\}$ is a continuous $g$-frame for $\alpha,\beta \in \mathbb{C}$ with $|\alpha|^2+|\beta|^2>0.$
\end{prop}
\begin{proof}
For any $h\in \h$ we have
\begin{align*}
\int_\Omega \|(\Lambda_\omega L_1+\Theta_\omega L_2)h\|^2 d\mu(\omega)&=\|T_\Lambda^*L_1 h+T_\Theta^*L_2 h\|^2\\&=\|T_\Lambda^*L_1 h\|^2+\|T_\Theta^*L_2 h\|^2,
\end{align*}
and
\begin{align*}
\big(B_\Lambda \|L_1\|^2+B_\Theta \|L_2\|^2\big)\|h\|^2 &\ge\|T_\Lambda^*L_1 h\|^ 2+\|T_\Theta^*L_2 h\|^ 2\\&\ge \min\{A_\Lambda, A_\Theta\}.\big(\|L_1 h\|^2+\|L_2 h\|^2\big) \\&=\min\{A_\Lambda, A_\Theta\}.\big\langle (L_1^* L_1+L_2^* L_2)h, h \big\rangle\\&=A.\min\{A_\Lambda, A_\Theta\}\|h\|^2.
\end{align*}
By taking $L_1=\alpha I$ and $L_2=\beta I$ the particular case is obvious.
\end{proof}
\begin{prop}
Let $\Lambda=\{\Lambda_\omega \in B(\h, \K_{\omega}) : \omega \in \Omega \}$ and $\Theta=\{\Theta_\omega \in B(\h, \K_{\omega}) : \omega \in \Omega \}$ be strongly disjoint Parseval continuous $g$-frames and let $L_1,L_2\in B(\h)$. Then $L_1^*L_1+L_2^*L_2=AI$ for some $A>0$ if and only if $\Lambda L_1+\Theta L_2=\{\Lambda_\omega L_1+\Theta_\omega L_2\in B(\h,\K_\omega) : \omega\in \Omega\}$ is a tight continuous $g$-frame with bound $A$. In particular, $\alpha\Lambda+\beta\Theta=\{\alpha\Lambda_\omega+\beta\Theta_\omega\in B(\h,\K_\omega) : \omega \in \Omega\}$ is a tight continuous $g$-frame if and only if $|\alpha|^2+|\beta|^2>0$ for $\alpha,\beta \in \mathbb{C}.$
\end{prop}
\begin{proof}
For any $h\in \h$ we have
\begin{align*}
\int_\Omega \|(\Lambda_\omega L_1+\Theta_\omega L_2)h\|^2 d\mu(\omega)
 &=\|T_\Lambda^*L_1 h+T_\Theta^*L_2 h\|^2
 \\&=\|T_\Lambda^*L_1 h\|^ 2+\|T_\Theta^*L_2 h\|^ 2
 \\&=\|L_1 h\|^ 2+\|L_2 h\|^2
 \\&=\big\langle (L_1^*L_1+L_2^*L_2)h,h\big\rangle
=A\|h\|^2.
\end{align*}
In particular, it is sufficient to take $L_1=\alpha I$ and $L_2=\beta I$.
\end{proof}
 \begin{cor}
Let $\Lambda=\{\Lambda_\omega \in B(\h, \K_{\omega}) : \omega \in \Omega \}$ and $\Theta=\{\Theta_\omega \in B(\h, \K_{\omega}) : \omega \in \Omega \}$ be strongly disjoint Parseval continuous $g$-frames and let $L_1,L_2\in B(\h).$ Then $L_1^*L_1+L_2^*L_2=I$ if and only if $\Lambda L_1+\Theta L_2=\{\Lambda_\omega L_1+\Theta_\omega L_2\in B(\h,\K_\omega) : \omega\in \Omega\}$ is a Parseval continuous $g$-frame. In particular, $\alpha\Lambda+\beta\Theta=\{\alpha\Lambda_\omega+\beta\Theta_\omega\in B(\h,\K_\omega) : \omega \in \Omega\}$ is a Parseval continuous $g$-frame if and only if $|\alpha|^2+|\beta|^2=1$ for $\alpha,\beta \in \mathbb{C}$.
\end{cor}
Now, to get a dual continuous $g$-frames by strongly disjoint continuous $g$-frames we generalize results of \cite{ostad2} and \cite{newnewnew}.
\begin{prop}\label{new}
Let $\Lambda=\{\Lambda_\omega \in B(\h, \K_{\omega}) : \omega \in \Omega \}$ and $\Psi=\{\Psi_\omega \in B(\K, \K_{\omega}) : \omega \in \Omega \}$ be duals of continuous $g$-frames $\Theta=\{\Theta_\omega \in B(\h,\K_\omega) : \omega \in \Omega \}$ and $\Phi=\{\Phi_\omega \in B(\K,\K_\omega) : \omega \in \Omega \}$, respectively. If $\Lambda$, $\Phi$ and $\Theta$, $\Psi$  are strongly disjoint. Then $\Gamma=\{\Gamma_\omega \in B(\h\oplus \K,\K_\omega) : \omega \in \Omega \}$  and $\Delta=\{\Delta_\omega \in B(\h\oplus \K,\K_\omega) : \omega \in \Omega \}$ are dual continuous $g$-frames, where
\begin{eqnarray*}
\Gamma_\omega(h\oplus k)=\Lambda_\omega h+\Psi_\omega k,\quad\Delta_\omega(h\oplus k)=\Theta_\omega h+\Phi_\omega k,
\end{eqnarray*}
for all $\omega\in \Omega$ and for all $h\in \h$, $k\in \K.$
\end{prop}
\begin{proof}
It is clear that $\Gamma$ and $\Delta$ are continuous $g$-Bessel families for $\h\oplus \K$. For any $h_1,h_2\in\h$ and $k_1,k_2\in \K$, we have
\begin{align*}
\int_\Omega\Big\langle \Gamma_\omega(h_1\oplus k_1)&,\Delta_\omega(h_2\oplus k_2)\Big\rangle d\mu(\omega)\\&=\int_\Omega\langle\Lambda_\omega h_1,\Theta_\omega h_2 \rangle d\mu(\omega)+\int_\Omega \langle\Lambda_\omega h_1,\Phi_\omega k_2 \rangle d\mu(\omega)
\\&+\int_\Omega\langle\Psi_\omega k_1,\Theta_\omega h_2 \rangle d\mu(\omega) +\int_\Omega\langle\Psi_\omega k_1,\Phi_\omega k_2 \rangle d\mu(\omega)
\\&=\langle h_1,h_2\rangle+\langle k_1,k_2 \rangle
\\&=\langle h_1\oplus k_1,h_2\oplus k_2 \rangle.
\end{align*}
Thus by Proposition 3.2 of \cite{ostad}, $\Gamma$ and $\Delta$ are dual continuous $g$-frames for $\h\oplus \K$.
\end{proof}
\begin{exa}
Let $F:\Omega\rightarrow \h$ and $G:\Omega\rightarrow \K$ be two continuous frames. We define two families of bounded operators $\Lambda=\{\Lambda_{\omega}\in B(\h,\mathbb{C}^2):\omega\in\Omega\}$ and $\Theta=\{\Theta_\omega\in B(\h,\mathbb{C}^2):\omega\in \Omega\}$ where
\begin{eqnarray*}
\Lambda_{\omega} f=\big(\langle f,F(\omega)\rangle_\h,0\big)\quad, \Theta_{\omega} f=\big(\langle S_F^{-1}f,F(\omega)\rangle_\h,0\big),
\end{eqnarray*}
for all $f\in\h$, $\omega\in\Omega.$ It is obvious that $\Lambda$ and $\Theta$ are continuous $g$-frames. Also for any $f,g\in \h$
\begin{eqnarray*}
\int_\Omega \langle \Theta_\omega f, \Lambda_\omega g\rangle d\mu(\omega)
=\int_\Omega \big\langle f, S^{-1}_F F(\omega)\big\rangle \langle F(\omega),g\rangle d\mu(\omega)
=\langle f,g \rangle.
\end{eqnarray*}
So, $\Lambda$ and $\Theta$ are dual continuous $g$-frames. We also define two families of bounded operators $\Phi=\{\Phi_\omega\in B(\K,\mathbb{C}^2):\omega\in \Omega\}$ and $\Psi=\{\Psi_{\omega}\in B(\K,\mathbb{C}^2):\omega\in\Omega\}$ where
\begin{eqnarray*}
\Phi_{\omega} g=\big(0,\langle S_G^{-1}g ,G(\omega)\rangle_\K\big),\quad \Psi_{\omega} g=\big(0,\langle g,G(\omega)\rangle_\K\big),
\end{eqnarray*}
for all $g\in \K$, $\omega\in\Omega.$ Similarly, $\Psi$ and $\Phi$ are dual continuous $g$-frames.
On the other hand, for any $f\in \h$  and $g\in\K$
\begin{eqnarray*}
\langle T^*_\Lambda f,T^*_{\Phi} g\rangle=\int_\Omega\langle \Lambda_\omega f,\Phi_\omega g\rangle d\mu(\omega)=0,
\end{eqnarray*}
so $\Lambda$ and $\Phi$ are strongly disjoint. Also, $\Theta$ and $\Psi$ are strongly disjoint.
Let us consider
\begin{align*}
\Gamma_\omega:\h\oplus\K\rightarrow\mathbb{C}^2,\quad\Gamma_{\omega}(f\oplus g)=\big(\langle f,F(\omega)\rangle_\h,\langle g,G(\omega)\rangle_\K\big),
\end{align*}
and
\begin{align*}
\Delta_\omega:\h\oplus\K\rightarrow\mathbb{C}^2,\quad\Delta_{\omega}(f\oplus g)=\big(\langle S_F^{-1}f,F(\omega)\rangle_\h,\langle S_G^{-1}g,G(\omega)\rangle_\K\big),
\end{align*}
for all $f\in \h$ and $g\in \K.$ Then
\begin{align*}
\int_\Omega \big\langle\Gamma_\omega(h_1\oplus k_1)&,\Delta_\omega(h_2\oplus k_2)\big\rangle d\mu(\omega)
\\&=\int_\Omega\langle h_1,F(\omega)\rangle_\h \langle S^{-1}_F F(\omega),  h_2\rangle_\h d\mu(\omega)
\\&+\int_\Omega\langle k_1,G(\omega)\rangle_\K \langle S^{-1}_G G(\omega), k_2\rangle_\K d\mu(\omega)
\\&=\langle h_1,h_2 \rangle+\langle k_1,k_2 \rangle
\\&=\langle h_1\oplus k_1,h_2\oplus k_2\rangle.
\end{align*}
Which Proposition \ref{new} confirm this result.
\end{exa}
\begin{prop}\label{guoguo}
Let $\Lambda=\{\Lambda_\omega\in B(\h,\K_\omega):\omega\in\Omega\}$ and $\Theta=\{\Theta_\omega\in B(\h,\K_\omega):\omega\in\Omega\}$ be two strongly disjoint continuous $g$-frames and $L_1,L_2\in B(\h).$ If $L_1$ is surjective, then
$\widetilde{\Lambda}L_1^\dagger=\{\Lambda_\omega S_\Lambda^{-1} L_1^\dagger\in B(\h,\K_\omega):\omega\in\Omega\}$ is a dual for both 
$\Lambda L_1^*=\{\Lambda_\omega L_1^*\in B(\h,\K_\omega):\omega\in\Omega\}$ and
 $\Lambda L_1^*+\Theta L_2^*=\{\Lambda_\omega L_1^*+\Theta_\omega L_2^*\in B(\h,\K_\omega):\omega\in\Omega\}$.
\end{prop}
\begin{proof}
It is obvious that $\widetilde{\Lambda}L_1^\dagger,$ $\Lambda L_1^*$ and $\Lambda L_1^*+\Theta L_2^*$ are continuous $g$-Bessel families. Since $L_1$ is surjective, then by Lemma \ref{snew} there exist $L_1^t\in B(\h,\K_\omega),$ such that $L_1L_1^t=I.$ For any $h,k\in\h$ we have 
\begin{align*}
\int_\Omega\big\langle(\Lambda_\omega L_1^*+\Theta_\omega L_2^*)h&,\Lambda_\omega S_\Lambda^{-1} L_1^\dagger k\big\rangle d\mu(\omega)
\\&=\int_\Omega\langle\Lambda_\omega L_1^*h,\Lambda_\omega S_\Lambda^{-1} L_1^\dagger k\rangle d\mu(\omega)
\\&+\int_\Omega\langle\Theta_\omega L_2^*h,\Lambda_\omega S_\Lambda^{-1} L_1^\dagger k\rangle d\mu(\omega)
\\&=\langle L_1^*h,L_1^\dagger k\rangle+\langle T_\Theta^* L_2^*h,T_\Lambda^* S_\Lambda^{-1} L_1^\dagger k\rangle
\\&=\langle h,L_1 L_1^\dagger k\rangle+0
=\langle h,k\rangle+0=\langle h,k\rangle.
\end{align*}
And also we have 
\begin{align*}
\int_\Omega\langle\Lambda_\omega L_1^*f,\Lambda_\omega S_\Lambda^{-1} L_1^\dagger g\rangle d\mu(\omega)
=\langle L_1^*f, L_1^\dagger g\rangle=\langle f,L_1 L_1^\dagger g\rangle=\langle f,g\rangle.
\end{align*}
Therefore the proof is completed.
\end{proof}
\begin{cor}
Let $\Lambda=\{\Lambda_\omega\in B(\h,\K_\omega):\omega\in\Omega\}$ and $\Theta=\{\Theta_\omega\in B(\h,\K_\omega):\omega\in\Omega\}$ be two strongly disjoint continuous $g$-frames. Then
$\widetilde{\Lambda}=\{\Lambda_\omega S_\Lambda^{-1}\in B(\h,\K_\omega):\omega\in\Omega\}$ is a dual for both 
$\Lambda=\{\Lambda_\omega\in B(\h,\K_\omega):\omega\in\Omega\}$ and
 $\Lambda+\Theta=\{\Lambda_\omega+\Theta_\omega\in B(\h,\K_\omega):\omega\in\Omega\}$.
\end{cor}

\begin{proof}
By considering $L_1=L_2=I$ in Proposition \ref{guoguo} the proof is completed.
\end{proof}
\section{{\textbf Some results related to Riesz-type continuous $g$-frames}}
In this section by generalizing some results of \cite{aaa}, we get some equivalet conditions for Riesz-type continuous $g$-frames.
\begin{thm}\label{thm9}
Let $\Lambda=\{\Lambda_{\omega}\in B(\h,\K_\omega):\omega\in \Omega\}$ be a
continuous $g$-frame. Then the following are equivalent:
\begin{enumerate}
\item[(i)] ${\Lambda}$  is a Riesz-type continuous g-frame.
\item[(ii)] There exist
constants $A,B>0$ such that
\begin{equation}\label{gframe1}
A\|\phi\|^{2}\leq \|T_{\Lambda}\phi\|^{2}\ \leq
B\|\phi\|^{2}, \quad \phi\in
\widehat \K.
\end{equation}
\item[(iii)] If
\begin{eqnarray*}\label{gframe2}
\int_{\Omega}\langle \Lambda_{\omega} ^ {*} \phi (\omega),f
\rangle d\mu(\omega)= 0
\end{eqnarray*}
for some $\phi\in
\widehat \K$ and for
any $f\in \h$, then $\phi=0$.
\end{enumerate}
\end{thm}
\begin{proof}
$(i) \Rightarrow (ii) $ By Proposition \ref{combi},
it remains to prove the left-hand inequality in (\ref{gframe1}). By Theorem
\ref{riezs}, for any $\phi\in
\widehat \K$, there exist
$f\in \h$ such that $T^{*}_{\Lambda}f=\phi$. Then
\begin{align*}
\|\phi\|^{4}=\Big(\int_{\Omega}\|\Lambda_\omega(f)\|^{2}d\mu(\omega)\Big)^2
&=|{\langle S_{\Lambda} f,f \rangle}|^{2}
\\&\leq \|S_{\Lambda} f\|^{2} \| f\|^{2}
\\&\leq \frac {1}{A_\Lambda} \|S_{\Lambda} f\|^{2}\int_{\Omega}\|\Lambda_\omega f\|^{2}d\mu(\omega),
\end{align*}
and hence
\begin{align*}
A_\Lambda\|\phi\|^{2} \leq \|S_{\Lambda} f\|^{2}=\|T_{\Lambda}
T^{*}_{\Lambda}f\|^{2}=\|T_{\Lambda} \phi\|^{2}.
\end{align*}
$(ii) \Rightarrow (iii) $ Let  for some $\phi\in \widehat \K$ and any $f\in \h,$ we have
\begin{eqnarray*}
\langle T_\Lambda\phi,f \rangle=\int_{\omega}\langle \Lambda_\omega^* \phi(\omega),f\rangle d\mu(\omega)=0.
\end{eqnarray*}
 Then $T_\Lambda \phi=0$ and by inequality (\ref{gframe1}), $\phi=0$.
 \\$(iii) \Rightarrow (i)$ Since $\Lambda$ is a continuous $g$-frame, $T_{\Lambda}$ is onto and by (iii) $T_{\Lambda}$ is one to one, so $T_{\Lambda}$ is invertible. Consequently, $ T_{\Lambda} ^{*}$ is invertible. Therefore, by Theorem \ref{riezs}, the proof is completed.
\end{proof}
Let $\Lambda=\{\Lambda_\omega \in B(\h, \K_{\omega}) : \omega \in \Omega \}$ and $\Theta=\{\Theta_\omega \in B(\K,\K_\omega):\omega \in \Omega \}$ be two continuous $g$-Bessel families. Consider the well defined operator $S_{\Theta\Lambda}:\h\rightarrow\K$, $S_{\Theta\Lambda}=T_\Theta T^*_\Lambda.$ Then
\begin{eqnarray*}
\langle S_{\Theta\Lambda}f,g \rangle=\int_\Omega \langle \Lambda_\omega f, \Theta_\omega g \rangle d\mu(\omega),\quad  f\in\h,g\in\K,
\end{eqnarray*}
and $S^*_{\Theta\Lambda}=S_{\Lambda\Theta}.$
\begin{thm}
Let $\Lambda=\{\Lambda_{\omega}\in B(\h,\K_\omega):\omega\in \Omega\}$ and
$\Theta=\{\Theta_{\omega}\in B(\h,\K_\omega):\omega\in\Omega\}$ be two continuous
g-Bessel family such that $S_{\Lambda\Theta}=I_\h.$
Assume that $L_{1},L_{2}:\h\rightarrow \h$ are bounded linear
operators so that $L_{1}^{*}L_{2}=I$. Then the following
statements are equivalent:
\\(i) $\Gamma=\{\Lambda_{\omega}L_{1}+\Theta_{\omega}L_{2}\in B(\h,\K_\omega):\omega\in\Omega\}$ is a Riesz-type continuous g-frame.
\\(ii) The operator $T_{\Lambda}^{*}L_{1}+T_{\Theta}^{*}L_{2}$ is
surjective.
\\(iii) There exists a constant $M>0$ such that
\begin{eqnarray*}
M\|\phi\|^{2}\leq \|(L_{1}^{*} T_{\Lambda}+L_{2}^{*}
T_{\Theta})\phi\|^{2},\quad\phi\in
\widehat \K.
\end{eqnarray*}
\end{thm}
\begin{proof}
For any $f\in \h$ we have
\begin{align*}
\int_{\Omega}\|(\Lambda_\omega L_1&+\Theta_\omega L_2)f\| ^{2}
d\mu(\omega)
\\&=\int_{\Omega}\|\Lambda_\omega L_1 f\|^{2}
d\mu(\omega)+\int\big\langle{\Lambda_\omega L_1 f},{\Theta_\omega L_2 f}\big\rangle d\mu(\omega)
\\&+\int\big\langle {\Theta_\omega L_2 f},{\Lambda_\omega L_1 f}\big\rangle
d\mu(\omega)+\int_{\Omega}\|\Theta_\omega L_2 f\|^{2} d\mu(\omega)
\\&=\int_{\Omega}\|\Lambda_\omega L_1 f\|^{2}
d\mu(\omega)+2\|f\|^2+\int_{\Omega}\|\Theta_\omega L_2 f\|^{2}d\mu(\omega).
\end{align*}
So
\begin{align*}
 2\|f \|^2& \leq \int_{\Omega}\|(\Lambda_\omega
L_1+\Theta_\omega L_2)f\|^{2} d\mu(\omega) 
\\&\leq \big(B_\Lambda\|L_1\|^2+2+ B_\Theta \|L_2\|^2\big)\|f \|^2.
\end{align*}
Hence $\Gamma$ is a continuous $g$-frame. On the other hand
\begin{align*}
T_\Gamma^*=T_\Lambda^* L_1 +T_\Theta^* L_2.
\end{align*}
By Theorem \ref{riezs}, (i) and (ii) are equivalent.
\\$(i)\Leftrightarrow (iii)$ It is concluded by Theorem \ref{thm9} and $T_\Gamma= L_1^* T_\Lambda+L_2^* T_\Theta.$
\end{proof}
\begin{prop}
Let $(\Omega,\mu)$ be a measure space, that $\mu$ is $\sigma$-finite and $\Lambda=\{\Lambda_\omega\in B(\h, \K_\omega):\omega \in \Omega\}$ is a continuous g-frame. Suppose that
$\Theta=\{\Theta_\omega\in B(\K, \K_\omega):\omega \in \Omega\}$ is a continuous
g-Bessel family. If $S_{\Theta\Lambda}$ is surjective, then
$\Theta$ is a continuous g-frame. If $\Theta$ is a continuous $g$-frame and $\Lambda$ is a
Riesz-type continuous g-frame then $S_{\Theta\Lambda}$ is surjective.
\end{prop}
\begin{proof}
Since $S_{\Theta\Lambda}$ is surjective,
it follows that $T_\Theta$ is surjective. On the other hand, by Proposition \ref{combi}, $T_\Theta$ is bounded. Hence by
Theorem \ref{ctf}, $\Theta$ is a continuous g-frame.
\\If $\Lambda$ is a Riesz-type continuous $g$-frame and $\Theta$ is a continuous $g$-frame then by Theorems \ref{riezs} and \ref{ctf}, $T_\Lambda^*$ and $T_\Theta$ are surjective. So, $S_{\Theta\Lambda}$ is surjective.
\end{proof}
\begin{thm}
Let $\Lambda=\{\Lambda_\omega\in B(\h, \K_\omega): \omega\in\Omega\}$ be a continuous
g-frame and $\Theta=\{\Theta_\omega \in B(\K, \K_\omega): \omega\in\Omega\}$ be a continuous g-Bessel family. Suppose that there exists a number
$\lambda$ with $0<\lambda<A_\Lambda$ such that
\begin{equation}
\|S_{\Theta\Lambda} f-S_\Lambda f\|\leq\lambda\|f\|,\quad f\in \h.
\end{equation}
Then $\Lambda$ is a Riesz-type continuous g-frame if and only if $\Theta$
is a Riesz-type continuous g-frame.
\end{thm}
\begin{proof}
 For all $f\in \h$ we have
\begin{align*}
 \|S_{\Theta\Lambda}f\|=\|S_{\Theta\Lambda}f-S_\Lambda
f+S_\Lambda f\|&\geq\|S_\Lambda f\|-\|S_{\Theta\Lambda}f-S_\Lambda f\|
\\
&\geq(A_\Lambda-\lambda)\|f\|.
\end{align*}
Then $S_{\Theta\Lambda}$ is injective with closed range. On the other hand,
\begin{eqnarray*}
\|S_{\Lambda\Theta}f-S_\Lambda f\| \leq\|(S_{\Theta\Lambda}-S_\Lambda)^*\|\|f\|\leq\lambda\|f\|.
\end{eqnarray*}
 So $S_{\Lambda\Theta}$ is also injective with closed range.Therefore
\begin{eqnarray*}
Range S_{\Theta\Lambda}=ker(S_{\Lambda\Theta})^\bot=\h
\end{eqnarray*}
and $Range S_{\Lambda\Theta}=\h.$
Thus, $S_{\Theta\Lambda}$ and $S_{\Lambda\Theta}$ are
invertible. Hence, $T_\Lambda^*$ is invertible if and only if
$T_\Theta^*$ is invertible. Then by Theorem \ref{riezs} the proof is completed.
\end{proof}

$$\textbf{Acknowledgment:}$$

\end{document}